\documentclass[10pt,a4paper]{amsart}
\usepackage{amssymb,amsmath,amsfonts}
\usepackage{tikz}
\headsep=1.2500cm \numberwithin{equation}{section}
\newtheorem{Theorem}{Theorem}[section]
\newtheorem{Proposition}[Theorem]{Proposition}
\newtheorem{cor}[Theorem]{Corollary}
\newtheorem{lemma}[Theorem]{Lemma}
\theoremstyle{remark}
\newtheorem{Definition}[Theorem]{Definition}
\newtheorem{Example}[Theorem]{Example}

\newtheorem{Remark}[Theorem]{Remark}

\begin{document}

\title{$WAP$-biprojectivity of the enveloping dual Banach algebras}

\author[S. F. Shariati]{S. F. Shariati}

\address{Faculty of Mathematics and Computer Science,
	Amirkabir University of Technology, 424 Hafez Avenue, 15914
	Tehran, Iran.}

\email{f.Shariati@aut.ac.ir}

\author[A. Pourabbas]{A. Pourabbas}
\email{arpabbas@aut.ac.ir}

\author[A. Sahami]{A. Sahami}
\address{Department of Mathematics Faculty of Basic Science, Ilam University, P.O. Box 69315-516 Ilam, Iran.}
\email{a.sahami@ilam.ac.ir}




\begin{abstract}
In this paper, we introduce a new notion of biprojectivity, called $WAP$-biprojectivity for $F(\mathcal{A})$, the enveloping dual Banach algebra associated to a Banach algebra $\mathcal{A}$. We find some relations between  Connes biprojectivity, Connes amenability and this new notion. We show that, for a given dual Banach algebra $\mathcal{A}$, if $F(\mathcal{A})$ is Connes amenable, then $\mathcal{A}$ is Connes amenable.

 For an infinite commutative compact group $G$, we show that the convolution Banach algebra $F(L^2(G))$ is not $WAP$-biprojective. Finally, we provide some examples of the enveloping dual Banach algebras and we study their $WAP$-biprojectivity and Connes amenability. 
\end{abstract}

\maketitle

\section{Introduction and Preliminaries}
Biprojectivity is one of the most important notions in Banach homology. In fact a Banach algebra $\mathcal{A}$ is biprojective if there exists a bounded $\mathcal{A}$-bimodule morphism $\rho:\mathcal{A}\rightarrow \mathcal{A}\hat{\otimes} \mathcal{A}$  such that $\pi_{\mathcal{A}}\circ\rho(a)=a,$ for every $a\in \mathcal{A}$.
It is well-known that the measure algebra $M(G)$ on a locally compact group $G$ is biprojective if and only if $G$ is finite, for more details see \cite{Runde:2002}. 

 There exists a class of Banach algebras which is called dual Banach algebras.  This category of Banach algebras were defined by Runde \cite{Runde:2001}. It is clear that every Banach algebra is not always dual Banach algebra but recently Choi, {\it et
al.} showed that there exists a dual Banach algebra associated to an arbitrary Banach algebra  which is called the enveloping dual Banach algebra \cite{Choi:2014}. Indeed,  
 let $\mathcal{A}$ be a Banach algebra and let $E$ be a Banach $\mathcal{A}$-bimodule. An element $x\in{E}$ is called weakly almost periodic if the module maps $\mathcal{A}\rightarrow{E}$; $a\mapsto{a\cdot{x}}$ and $a\mapsto{x\cdot{a}}$ are weakly compact. The set of all weakly almost periodic elements of $E$ is denoted by $WAP(E)$ which is a norm closed sub-bimodule of $E$ \cite[Definition 4.1]{Runde:2004}. 
For a Banach algebra $\mathcal{A}$, we write $F(\mathcal{A})_{\ast}$ for the $\mathcal{A}$-bimodule $WAP(\mathcal{A}^{\ast})$ which is the left introverted subspace of $\mathcal{A}^{\ast}$ in the sense of \cite[\textsection1]{Lau:97}. Runde observed that $F(\mathcal{A})=WAP(\mathcal{A}^{\ast})^{\ast}$ is a dual Banach algebra with the first Arens product inherited from $\mathcal{A}^{\ast\ast}$. He also showed that $F(\mathcal{A})$ is a canonical dual Banach algebra associated to $\mathcal{A}$ \cite[Theorem 4.10]{Runde:2004}. Choi, {\it et al.} in \cite{Choi:2014} called $F(\mathcal{A})$ the enveloping dual Banach algebra associated to $\mathcal{A}$. They  showed that if $\mathcal{A}$ is a Banach algebra and $X$ is a Banach $\mathcal{A}$-bimodule, then $F_{\mathcal{A}}(X)=WAP(X^{\ast})^{\ast}$ is a normal dual $F(\mathcal{A})$-bimodule \cite[Theorem 4.3]{Choi:2014}. Also they studied the Connes amenability of $F(\mathcal{A})$. Indeed they showed that for a given Banach algebra $\mathcal{A}$, the dual Banach algebra $F(\mathcal{A})$ is Connes amenable if and only if $\mathcal{A}$ admits a $WAP$-virtual diagonal \cite[Theorem 6.12]{Choi:2014}.

Motivated by these results, first we introduce the notion of $WAP$-biprojectivity for the enveloping dual Banach algebra associated to a Banach algebra $\mathcal{A}$.  Next for a Banach algebra $\mathcal{A}$ we investigate the relation between $WAP$-biprojectivity  of $F(\mathcal{A})$ with biprojectivity of $\mathcal{A}$ and also for a a dual Banach algebra $\mathcal{A}$ we study the relation between $WAP$-biprojectivity  of $F(\mathcal{A})$ with Connes biprojectivity of $\mathcal{A}$. We conclude that the Connes amenability of $F(\mathcal{A})$ implies the Connes amenability of   dual Banach algebra $\mathcal{A}$. 
We show that for a locally compact group $G$, if $F(M(G))$ is $WAP$-biprojective, then $G$ is amenable and for an infinite commutative compact group $G$ we show that 
$F(L^{2}(G))$ is not  $WAP$-biprojective.
Finally, we provide some examples of the enveloping dual Banach algebras associated to the certain Banach algebras and we study  their $WAP$-biprojectivity and Connes amenability.

 Let $\mathcal{A}$ be a Banach algebra. An $\mathcal{A}$-bimodule $E$ is called dual if there is a closed submodule ${E}_{\ast}$ of ${E}^{\ast}$ such that $E=(E_{\ast})^{\ast}$. The Banach algebra $\mathcal{A}$ is called dual if it is dual as a Banach $\mathcal{A}$-bimodule. A dual Banach $\mathcal{A}$-bimodule $E$ is normal, if for each $x\in{E}$ the module maps $\mathcal{A}\longrightarrow{E}$; ${a}\mapsto{a}\cdot{x}$ and ${a}\mapsto{x}\cdot{a}$ are $wk^\ast$-$wk^\ast$ continuous. Let $\mathcal{A}$ be a Banach algebra and let $E$ be a Banach $\mathcal{A}$-bimodule. A bounded linear map $D:\mathcal{A}\longrightarrow{E} $ is called a bounded derivation if for every $a,b\in\mathcal{A}$, $D(ab)=a\cdot{D(b)}+D(a)\cdot{b}$. A derivation $D:\mathcal{A}\longrightarrow{E}$ is called inner if there exists an element $x$ in $E$ such that $D(a)=a\cdot{x}-x\cdot{a}$ ($a\in{\mathcal{A}}$). A dual Banach algebra $\mathcal{A}$ is said to be Connes amenable if for every normal dual Banach $\mathcal{A}$-bimodule $E$, every $wk^\ast$-continuous derivation $D:\mathcal{A}\longrightarrow{E}$ is inner. For a given dual Banach algebra $\mathcal{A}$ and a Banach $\mathcal{A}$-bimodule $E$, $\sigma{wc}(E)$ denotes the set of all elements $x\in{E}$ such that the module maps $\mathcal{A}\rightarrow{E}$; ${a}\mapsto{a}\cdot{x}$ and ${a}\mapsto{x}\cdot{a}$
 are $wk^\ast$-$wk$-continuous, one can see that, it is a closed submodule of $E$ (see \cite{Runde:2001} and \cite{Runde:2004} for more details). For a given Banach algebra $\mathcal{A}$, consider the product morphism $\pi_{\mathcal{A}}:\mathcal{A}\hat{\otimes}\mathcal{A}\longrightarrow\mathcal{A}$ given by $\pi_{\mathcal{A}}(a\otimes{b})=ab$ for every $a,b\in{\mathcal{A}}$.  Since $\sigma{wc}(\mathcal{A}_{\ast})=\mathcal{A}_{\ast}$, the adjoint of $\pi_{\mathcal{A}}$ maps $\mathcal{A}_{\ast}$ into $\sigma{wc}(\mathcal{A}\hat{\otimes}\mathcal{A})^{\ast}$. Therefore, $\pi_{\mathcal{A}}^{\ast\ast}$ drops to an $\mathcal{A}$-bimodule morphism $\pi_{\sigma{wc}}:(\sigma{wc}(\mathcal{A}\hat{\otimes}\mathcal{A})^{\ast})^{\ast}\longrightarrow\mathcal{A}$. ٍEvery element $M\in{(\sigma{wc}(\mathcal{A}\hat{\otimes}\mathcal{A})^{\ast})^*}$ satisfying
 \begin{center}
 	$a\cdot{M}=M\cdot{a}\quad$ and $\quad{a}\pi_{\sigma{wc}}M=a\quad(a\in{\mathcal{A}})$,
 \end{center}
 is called a $\sigma{wc}$-virtual diagonal for $\mathcal{A}$. Runde showed that a dual Banach algebra $\mathcal{A}$ is Connes amenable if and only if there is a $\sigma{wc}$-virtual diagonal for $\mathcal{A}$ \cite[Theorem 4.8]{Runde:2004}.\\
 Let $\Delta_{WAP}:{F}_{\mathcal{A}}(\mathcal{A}\hat{\otimes}\mathcal{A})\longrightarrow{F}(\mathcal{A})$ be the $wk^{\ast}$-$wk^{\ast}$ continuous $\mathcal{A}$-bimodule map induced by $\pi_\mathcal{A}:\mathcal{A}\hat{\otimes}\mathcal{A}\longrightarrow\mathcal{A}$. Note that $\Delta_{WAP}$ is also an $F(\mathcal{A})$-bimodule map (see \cite[Corollary 5.2]{Choi:2014} for more details). Composing the canonical inclusion map $\mathcal{A}\hookrightarrow\mathcal{A}^{\ast\ast}$ with the adjoint of the inclusion map $F(\mathcal{A})_{\ast}\hookrightarrow\mathcal{A}^{\ast}$, we obtain a continuous homomorphism of Banach algebras $\eta_{\mathcal{A}}:\mathcal{A}\longrightarrow{F}(\mathcal{A})$ which has a $wk^{\ast}$-dense range. We write $\bar{a}$ instead of $\eta_{\mathcal{A}}(a)$ \cite[Definition 6.4]{Choi:2014}. Let $\mathcal{A}$ be a Banach algebra. An element $M\in{{F}_{\mathcal{A}}(\mathcal{A}\hat{\otimes}\mathcal{A})}$ is called a $WAP$-virtual diagonal for $\mathcal{A}$ if for every $a\in{\mathcal{A}}$
\begin{equation*}
a\cdot{M}=M\cdot{a}\quad \hbox{and} \quad\Delta_{WAP}(M)\cdot{a}=\bar{a}.
\end{equation*}
 The notion of $\varphi$-Connes amenability for a dual Banach algebra $\mathcal{A}$, where $\varphi$ is a ${wk}^{\ast}$-continuous multiplicative linear functional (character) on $\mathcal{A}$, was introduced by Mahmoodi  and some characterizations were given in \cite{Mahmoodi:2014}. We say that $\mathcal{A}$ is $\varphi$-Connes amenable if there exists a bounded linear functional $m$ on  $\sigma{wc}({\mathcal{A}}^{\ast})$ satisfying $m(\varphi)=1$ and $m(f\cdot{a})=\varphi(a)m(f)$ for every $a\in{\mathcal{A}}$ and $f\in{\sigma{wc}({\mathcal{A}}^{\ast})}$. The concept of $\varphi$-Connes amenability was characterized through vanishing of the cohomology group $\mathcal{H}^{1}_{wk^{*}}(\mathcal{A},E)$ for certain normal dual Banach $\mathcal{A}$-bimodule $E$.
By \cite[Theorem 2.2]{Mahmoodi:2014}, we conclude that every Connes amenable Banach algebra is $\varphi$-Connes amenable, where $\varphi$ is a ${wk}^{\ast}$-continuous character on $\mathcal{A}$.

At the following we  give the definition of our new notion:
\begin{Definition}\label{D5.1}
	Let $\mathcal{A}$ be a Banach algebra. Then $F(\mathcal{A})$ is called $WAP$-biprojective if there exists a $wk^{\ast}$-$wk^{\ast}$ continuous $\mathcal{A}$-bimodule morphism $\rho :{F}(\mathcal{A})\longrightarrow{F}_{\mathcal{A}}(\mathcal{A}\hat{\otimes}\mathcal{A})$ such that $\Delta_{WAP}\circ\rho=id_{F(\mathcal{A})}$. 
\end{Definition}
\section{$WAP$-biprojectivity of the enveloping dual Banach algebras}
In this section we study general property of the  $WAP$-biprojective Banach algebras and we investigate the relation of this notion with the other notions of Connes amenability on dual Banach algebras.

Let $\mathcal{A}$ be a Banach algebra. An $\mathcal{A}$-bimodule $X$ is called contractive if for every $x\in{X}$ and $a\in{\mathcal{A}}$
\begin{equation*}
\Vert{a}\cdot{x}\Vert\leq\Vert{a}\Vert\Vert{x}\Vert\quad\hbox{and}\quad\Vert{x}\cdot{a}\Vert\leq\Vert{x}\Vert\Vert{a}\Vert.
\end{equation*}
Following \cite[\textsection3]{Choi:2014}, let $\mathcal{A}$ be a Banach algebra and let $X$ be a contractive  $\mathcal{A}$-bimodule. Then  $\mathcal{A}\oplus_{\ltimes}X$ is called the triangular Banach algebra associated to $(\mathcal{A},X)$ equipped with $\ell^1$-norm and the product
\begin{equation*}
(a,x)\cdot(b,y)\colon=(ab,a\cdot{y}+x\cdot{b})\qquad(a,b\in{\mathcal{A}}, x,y\in{X}).
\end{equation*}
\begin{Remark}
For technical reasons, Choi,  {\it et al.} worked with bimodules and normal dual bimodules that are contractive \cite{Choi:2014}. Note that for a given Banach algebra $\mathcal{A}$, if $X$ is a Banach $\mathcal{A}$-bimodule, then by a standard renorming argument there exists a contractive $\mathcal{A}$-bimodule $Y$ which is isomorphism to $X$, moreover if $M$ is a normal dual $\mathcal{A}$-bimodule, where $\mathcal{A}$ is a dual Banach algebra, then there exists a contractive, normal dual $\mathcal{A}$-bimodule $N$ which is $wk^*$-isomorphism to $M$ \cite[Lemma 2.4]{Choi:2014}. So without loss of generality, it is possible to extract the results of \cite{Choi:2014} in terms of wider classes of bimodule without contractive condition (see \cite[\textsection2.1]{Choi:2014} for more details).
\end{Remark}
\begin{lemma}\label{l5.3}
	Let $\mathcal{A}$ be a Banach algebra and let $X$ be a Banach $\mathcal{A}$-bimodule. Then for every $a\in{\mathcal{A}}$, $\eta\in{F(\mathcal{A})}$ and $\psi\in{F_{\mathcal{A}}(X)}$ we have
	\begin{enumerate}
		\item [(i)] $a\cdot\eta=\bar{a}\square\eta \quad(\eta\cdot{a}=\eta\square\bar{a})$,
		\item[(ii)] $a\cdot\psi=\bar{a}\bullet\psi\quad(\psi\cdot{a}=\psi\bullet\bar{a})$,
	\end{enumerate}
	where $\square$ and $\bullet$ are the first Arens product in $F(\mathcal{A})$ and the module action of $F(\mathcal{A})$ on $F_{\mathcal{A}}(X)$, respectively.
\end{lemma}
\begin{proof}
	(i) For every $f\in{WAP(\mathcal{A}^{\ast})}$,
	\begin{equation*}
	\langle{f},\bar{a}\square\eta\rangle=\langle\eta\cdot{f},\bar{a}\rangle=\langle{a},\eta\cdot{f}\rangle=\langle{f}\cdot{a},\eta\rangle=\langle{f},a\cdot\eta\rangle\quad(a\in{\mathcal{A}}, \eta\in{F(\mathcal{A})}).
	\end{equation*}
	Also
	\begin{equation*}
	\langle{b},\bar{a}\cdot{f}\rangle=\langle{f}\cdot{b},\bar{a}\rangle=\langle{a},{f}\cdot{b}\rangle=\langle{b},a\cdot{f}\rangle\quad(a,b\in{\mathcal{A}}),
	\end{equation*}
	similarly
	\begin{equation*}
	\langle{f},\eta\square\bar{a}\rangle=\langle\bar{a}\cdot{f},\eta\rangle=\langle{a}\cdot{f},\eta\rangle=\langle{f},\eta\cdot{a}\rangle\quad(a\in{\mathcal{A}}, \eta\in{F(\mathcal{A})}).
	\end{equation*}
	(ii) According to \cite[Theorem 4.3]{Choi:2014},
	\begin{equation*}
	(0,\bar{a}\bullet\psi)=(\bar{a},0)\square(0,\psi)\quad\hbox{in }{F}(\mathcal{A}\oplus_\ltimes{X}).
	\end{equation*} 
	So for every $f\in{WAP(\mathcal{A}^{\ast})}$ and $g\in{WAP({X}^{\ast})}$ we have
	\begin{equation*}
	\begin{split}
	\langle(f,g),(\bar{a},0)\square(0,\psi)\rangle&=\langle(0,\psi)\cdot(f,g),(\overline{a,0})\rangle\\ &=\langle(a,0),(0,\psi)\cdot(f,g)\rangle\\&=\langle(f,g)\cdot(a,0),(0,\psi)\rangle\quad(a\in{\mathcal{A}},  \psi\in{F_{\mathcal{A}}(X)}).
	\end{split}
	\end{equation*}
	Also
	\begin{equation*}
	\begin{split}
	\langle(b\cdot x),(f,g)\cdot(a,0)\rangle&=\langle(ab,a\cdot x+0\cdot b),(f,g)\rangle=\langle ab,f\rangle+\langle a\cdot x,g \rangle\\&=\langle b,f\cdot a\rangle+\langle x,g\cdot a\rangle=\langle(b.x),(f\cdot a,g\cdot a)\rangle\quad(a,b\in{\mathcal{A}}, x\in{X}).
	\end{split}
	\end{equation*}
	It follows that
	\begin{equation*}
	\langle(f,g),(\bar{a},0)\square(0,\psi)\rangle=\langle(f\cdot{a},g\cdot{a}),(0,\psi)\rangle=\langle{f},a\cdot0\rangle+\langle{g},a\cdot\psi\rangle=\langle(f,g),(0,a\cdot\psi)\rangle.
	\end{equation*}
	The proof for the right action is similar.
\end{proof}

\begin{Remark}
	Consider the $\mathcal{A}$-bimodule morphism $\rho$ as in Definition \ref{D5.1}: 
	\begin{enumerate}
		\item [(i)] 	Since $\eta_{\mathcal{A}}:\mathcal{A}\longrightarrow{F}(\mathcal{A})$ has a $wk^{\ast}$-dense range, for every $\psi\in{F}(\mathcal{A})$ there exist a bounded net $(u_{\alpha})$ in ${\mathcal{A}}$ such that $\psi=wk^*$-$\lim\limits_{\alpha}\bar{u}_{\alpha}$. Also since $F(\mathcal{A})$ and ${F}_{\mathcal{A}}(\mathcal{A}\hat{\otimes}\mathcal{A})$ are normal as Banach $F(\mathcal{A})$-bimodules and $\rho$ is $wk^{\ast}$-$wk^{\ast}$ continuous, Lemma \ref{l5.3} implies that for every $\phi\in{F(\mathcal{A})}$ we have
		\begin{equation*}
		\begin{split}
		\psi\bullet\rho(\phi)&=wk^*\hbox{-}\lim\limits_{\alpha}\bar{u}_{\alpha}\bullet\rho(\phi)=wk^*\hbox{-}\lim\limits_{\alpha}(\bar{u}_{\alpha}\bullet\rho(\phi))=wk^*\hbox{-}\lim\limits_{\alpha}({u}_{\alpha}\cdot\rho(\phi))\\&=wk^*\hbox{-}\lim\limits_{\alpha}\rho({u}_{\alpha}\cdot\phi)=wk^*\hbox{-}\lim\limits_{\alpha}\rho(\bar{u}_{\alpha}\square\phi)=\rho(wk^*\hbox{-}\lim\limits_{\alpha}(\bar{u}_{\alpha}\square\phi))\\&=\rho(\psi\square\phi),
		\end{split}
		\end{equation*}
		where $\bullet$ denotes the module action of $F(\mathcal{A})$ on $F_{\mathcal{A}}(\mathcal{A}\hat{\otimes}\mathcal{A})$. So by similarity for the right action, $\rho$ is an $F(\mathcal{A})$-bimodule morphism and also $\mathcal{A}$-bimodule morphism.
		\item[(ii)] Using \cite[Corollary 3.1.12]{Rob:98}, $\rho$ is norm continuous, moreover $\rho$ is $wk^{\ast}$-$wk^{\ast}$ continuous.  
	\end{enumerate}  
\end{Remark}
Choi {\it et al.} \cite[Theorem 6.12]{Choi:2014} for a Banach algebra $\mathcal{A}$ showed that $F(\mathcal{A})$ is Connes-amenable if and only if $\mathcal{A}$ has a $WAP$-virtual diagonal. In the following theorem we extend an analogue result for $WAP$-biprojectivity:
\begin{Theorem}
	Let $\mathcal{A}$ be a Banach algebra. Then the followings are equivalent;
	\begin{enumerate}
		\item [(i)]$F(\mathcal{A})$ is $WAP$-biprojective with an identity,
		\item[(ii)] $\mathcal{A}$ has a $WAP$-virtual diagonal.
	\end{enumerate}
\end{Theorem}
\begin{proof}
	(i)$\Rightarrow$(ii) Suppose that $F(\mathcal{A})$ is $WAP$-biprojective with an identity $e$. Then there exists a $wk^{\ast}$-$wk^{\ast}$ continuous $\mathcal{A}$-bimodule morphism $\rho:{F}(\mathcal{A})\longrightarrow{F}_{\mathcal{A}}(\mathcal{A}\hat{\otimes}\mathcal{A})$ such that $\Delta_{WAP}\circ\rho=id_{F(\mathcal{A})}$. Let $M=\rho(e)$. Then $M$ is an element in ${F}_{\mathcal{A}}(\mathcal{A}\hat{\otimes}\mathcal{A})$ and by Lemma \ref{l5.3} (i) for every $a\in{\mathcal{A}}$, we have
	\begin{equation*}
	a\cdot{M}=a\cdot\rho(e)=\rho(a\cdot{e})=\rho(\bar{a}\square{e})=\rho(e\square\bar{a})=\rho(e\cdot{a})=\rho(e)\cdot{a}=M\cdot{a},
	\end{equation*}
	and
	\begin{equation*}
	\Delta_{WAP}(M)\cdot{a}=(\Delta_{WAP}\circ\rho(e))\cdot{a}=e\cdot{a}=e\square\bar{a}=\bar{a}.
	\end{equation*}
	So $M$ is a $WAP$-virtual diagonal for $F(\mathcal{A})$.\\
	(ii)$\Rightarrow$(i) Suppose that there exists a $WAP$-virtual diagonal $M$ for $\mathcal{A}$. We define                                                                                                              $\rho:{F}(\mathcal{A})\longrightarrow{F}_{\mathcal{A}}(\mathcal{A}\hat{\otimes}\mathcal{A})$ by $\rho(\eta)=\eta\bullet{M}$, for every $\eta\in{F(\mathcal{A})}$, where $\bullet$ denotes the module action of $F(\mathcal{A})$ on ${F}_{\mathcal{A}}(\mathcal{A}\hat{\otimes}\mathcal{A})$. Since ${F}_{\mathcal{A}}(\mathcal{A}\hat{\otimes}\mathcal{A})$ is a normal dual $F(\mathcal{A})$-bimodule \cite[Theorem 4.3]{Choi:2014}, $\rho$ is $wk^{\ast}$-$wk^{\ast}$ continuous. Using Lemma \ref{l5.3}, for every $a\in{\mathcal{A}}$ and $\eta\in{F(\mathcal{A})}$ we have
	\begin{equation*}
	a\cdot\rho(\eta)=a\cdot(\eta\bullet{M})=\bar{a}\bullet(\eta\bullet{M})=(\bar{a}\square \eta)\bullet{M}=(a\cdot\eta)\bullet{M}=\rho(a\cdot\eta),
	\end{equation*}
	On the other hand, since $M$ is a $WAP$-virtual diagonal, by \cite[Remark 6.5]{Choi:2014} we have
	\begin{equation*}
	\rho(\eta)\cdot{a}=(\eta\bullet{M})\cdot{a}=({M}\bullet\eta)\bullet\bar{a}={M}\bullet(\eta\square\bar{a})={M}\bullet(\eta\cdot{a})=(\eta\cdot{a})\bullet{M}=\rho(\eta\cdot{a}).
	\end{equation*}
	So $\rho$ is an $\mathcal{A}$-bimodule morphism. Since $\Delta_{WAP}$ is an  $F(\mathcal{A})$-bimodule morphism \cite[Corollary 5.2]{Choi:2014}, for every $\eta\in{F(\mathcal{A})}$ we have
	\begin{equation*}
	\Delta_{WAP}\circ\rho(\eta)=\Delta_{WAP}(\eta\bullet{M})=\eta\bullet\Delta_{WAP}(M)=\eta.
	\end{equation*}
	Therefore $F(\mathcal{A})$ is $WAP$-biprojective and it is easy to see that $\Delta_{WAP}(M)$ is an identity for $F(\mathcal{A})$ \cite[Remark 6.5]{Choi:2014}. 
\end{proof}
\begin{cor}\label{c2.5}
Let $\mathcal{A}$ be a Banach algebra. $F(\mathcal{A})$ is Connes amenable if and only if $F(\mathcal{A})$ is $WAP$-biprojective with an identity.  
\end{cor}
\begin{proof}
Applying \cite[Theorem 6.12]{Choi:2014} and the previous Theorem.	
\end{proof}
 A dual Banach algebra $\mathcal{A}$ is called Connes biprojective if there exists a bounded $\mathcal{A}$-bimodule morphism $\rho:\mathcal{A}\longrightarrow(\sigma{wc}(\mathcal{A}\hat{\otimes}\mathcal{A})^{\ast})^{\ast}$ such that $\pi_{\sigma{wc}}\circ\rho=id_{\mathcal{A}}$ . Shirinkalam and the second author showed that a dual Banach algebra $\mathcal{A}$ is Connes amenable if and only if $\mathcal{A}$ is Connes biprojective and has an identity, for more details see \cite{Shi:2016}.

\begin{Theorem}\label{Th 5.5}
	Let $\mathcal{A}$ be a Banach algebra. Then
	\begin{enumerate}
		\item [(i)] If $\mathcal{A}$ is biprojective, then $F(\mathcal{A})$ is $WAP$-biprojective.
		\item[(ii)] If $\mathcal{A}$ is a dual Banach algebra and $F(\mathcal{A})$ is $WAP$-biprojective, then $\mathcal{A}$ is Connes biprojective.
	\end{enumerate}
\end{Theorem}
\begin{proof}
	(i) Suppose that $\mathcal{A}$ is biprojective. Then there exists a bounded $\mathcal{A}$-bimodule morphism $\psi:\mathcal{A}\longrightarrow\mathcal{A}\hat{\otimes}\mathcal{A}$ such that $\psi$ is a right inverse for $\pi_{\mathcal{A}}$. By standard properties of weakly compact maps, it is easy to see that $\psi^{\ast}(WAP(\mathcal{A}\hat{\otimes}\mathcal{A})^{\ast})\subseteq{WAP}(\mathcal{A}^{\ast})$. Let $\rho=(\psi^{\ast}\vert_{WAP(\mathcal{A}\hat{\otimes}\mathcal{A})^{\ast}})^{\ast}$. We obtain a $wk^{\ast}$-$wk^{\ast}$-continuous $\mathcal{A}$-bimodule morphism $\rho:F(\mathcal{A})\longrightarrow{F}_{\mathcal{A}}(\mathcal{A}\hat{\otimes}\mathcal{A})$. According to the functor $F(-)$ from Banach algebras into the dual Banach algebras \cite[Remark 2.10]{Choi:2014}, we have $\rho=F(\psi)$ and $\Delta_{WAP}=F(\pi_\mathcal{A})$. Also both squares in the following diagram commute:
	\begin{center}
		\begin{tikzpicture}[node distance=2cm, auto]
		\node (A) {$\mathcal{A}$};
		\node(B) [right of=A] {$\mathcal{A}\hat{\otimes}\mathcal{A}$};
		\node(C) [right of=B] {$\mathcal{A}$};
		\node (D) [below of=A] {$F(\mathcal{A})$};
		\node(E) [right of=D] {$F_{\mathcal{A}}(\mathcal{A}\hat{\otimes}\mathcal{A})$};
		\node(F) [right of=E] {$F(\mathcal{A}),$};
		\draw[->](A) to node {$\psi$}(B);
		\draw[->](B) to node {$\pi_\mathcal{A}$}(C);
		\draw[->](A) to node [left] {$\eta_{\mathcal{A}}$}(D);
		\draw[->](B) to node [left] {$\eta_{\mathcal{A}\hat{\otimes}\mathcal{A}}$}(E);
		\draw[->](D) to node {$F(\psi)$}(E);
		\draw[->](E) to node {$F(\pi_\mathcal{A})$}(F);
		\draw[->](C) to node [left] {$\eta_{\mathcal{A}}$}(F);
	\end{tikzpicture}
	\end{center}
	therefore the outer rectangle commutes, that is, $F(\pi_\mathcal{A})\circ F(\psi)\circ\eta_{\mathcal{A}}=\eta_{\mathcal{A}}\circ\pi_\mathcal{A}\circ\psi$. On the other hand \cite[Corollary 5.2]{Choi:2014} implies that $F(\pi_\mathcal{A}\circ\psi):F(\mathcal{A})\rightarrow F(\mathcal{A})$ is a unique $wk^*$-$wk^*$ continuous map such that $F(\pi_\mathcal{A}\circ\psi)\circ\eta_{\mathcal{A}}=\eta_{\mathcal{A}}\circ\pi_\mathcal{A}\circ\psi$. So we have $F(\pi_\mathcal{A})\circ F(\psi)=F(\pi_\mathcal{A}\circ\psi)$. Thus
	\begin{equation*}
	{\Delta_{WAP}}\circ\rho=F(\pi_\mathcal{A})\circ F(\psi)=F(\pi_\mathcal{A}\circ\psi)=F(id_{\mathcal{A}})=id_{F(\mathcal{A})}.
	\end{equation*}
	So $F(\mathcal{A})$ is $WAP$-biprojective.\\
	(ii) Suppose that $\mathcal{A}$ is a dual Banach algebra and $F(\mathcal{A})$ is $WAP$-biprojective. Then there exists a bounded $\mathcal{A}$-bimodule morphism $\rho:F(\mathcal{A})\longrightarrow{F}_{\mathcal{A}}(\mathcal{A}\hat{\otimes}\mathcal{A})$ such that 
	\begin{equation}\label{e2.1}
	\Delta_{WAP}\circ\rho=id_{F(\mathcal{A})}.
	\end{equation}
Since $(\sigma{wc}(\mathcal{A}\hat{\otimes}\mathcal{A})^{\ast})^{\ast}$ is normal, \cite[Proposition 4.2]{Runde:2004} implies that $\sigma{wc}(\mathcal{A}\hat{\otimes}\mathcal{A})^{\ast}\subseteq{WAP}(\mathcal{A}\hat{\otimes}\mathcal{A})^{\ast}$. So there is a natural quotient map $q: {F}_{\mathcal{A}}(\mathcal{A}\hat{\otimes}\mathcal{A})\longrightarrow(\sigma{wc}(\mathcal{A}\hat{\otimes}\mathcal{A})^{\ast})^{\ast}$ which is defined by $q(u)=u\vert_{\sigma{wc}(\mathcal{A}\hat{\otimes}\mathcal{A})^{\ast}}$ for every $u\in{F_{\mathcal{A}}(\mathcal{A}\hat{\otimes}\mathcal{A})}$.  On the other hand, since $\mathcal{A}$ is a dual Banach algebra, $\mathcal{A}_{\ast}\subseteq{WAP}(\mathcal{A}^{\ast})$ \cite[Proposition 4.2]{Runde:2004}, there exists a quotient map $q^{\prime}:{F}(\mathcal{A})\longrightarrow{\mathcal{A}}$ which is defined by $q^{\prime}(\phi)=\phi\vert_{\mathcal{A}_{\ast}}$ for every $\phi\in{F(\mathcal{A})}$. So
\begin{equation}\label{e2.2}
q^{\prime}\circ\eta_{\mathcal{A}}=id_{\mathcal{A}}.
\end{equation}
  Also since $\Delta_{WAP}=(\pi^{\ast}\vert_{WAP(\mathcal{A}^{\ast})})^{\ast}$ and $\pi_{\sigma{wc}}=(\pi^{\ast}\vert_{\mathcal{A}_{\ast}})^{\ast}$, for every $u\in{F_{\mathcal{A}}(\mathcal{A}\hat{\otimes}\mathcal{A})}$ and $f\in{\mathcal{A}_{\ast}}$ we have
	\begin{equation*}
	\begin{split}
	\langle{f},q^{\prime}\circ\Delta_{WAP}(u)\rangle&=\langle f, (\Delta_{WAP}(u))\vert_{\mathcal{A}_{\ast}}\rangle=\langle{f},\Delta_{WAP}(u)\rangle\\&=\langle\pi^{\ast}\vert_{WAP(\mathcal{A}^{\ast})}(f),u\rangle=\langle\pi_{\mathcal{A}}^{\ast}(f),u\rangle,
\end{split}
	\end{equation*}
	and 
		\begin{equation*}
	\begin{split}
	\langle{f},\pi_{\sigma{wc}}\circ{q}(u)\rangle&=\langle\pi^{\ast}\vert_{\mathcal{A}_{\ast}}(f),q(u)\rangle=\langle\pi^{\ast}\vert_{\mathcal{A}_{\ast}}(f),u\vert_{\sigma{wc}({\mathcal{A}}\hat{\otimes}{\mathcal{A}})^{\ast}}\rangle\\&=\langle\pi^{\ast}\vert_{\mathcal{A}_{\ast}}(f),u\rangle=\langle\pi_{\mathcal{A}}^{\ast}(f),u\rangle.
	\end{split}
	\end{equation*}
	So for every $u\in{F_{\mathcal{A}}(\mathcal{A}\hat{\otimes}\mathcal{A})}$ we have $q^{\prime}\circ\Delta_{WAP}(u)=\pi_{\sigma{wc}}\circ{q}(u)$ as an element in $\mathcal{A}$. Then 
	\begin{equation}\label{e2.3}
	q^{\prime}\circ\Delta_{WAP}=\pi_{\sigma{wc}}\circ{q}.
	\end{equation}
Let $\tau=q\circ\rho\circ\eta_{\mathcal{A}}$. We obtain a bounded $\mathcal{A}$-bimodule morphism $\tau:\mathcal{A}\longrightarrow(\sigma{wc}(\mathcal{A}\hat{\otimes}\mathcal{A})^{\ast})^{\ast}$. Therefore 
(\ref{e2.1}), (\ref{e2.2}) and (\ref{e2.3}) imply that
	\begin{equation*}
	\pi_{\sigma{wc}}\circ\tau=\pi_{\sigma{wc}}\circ{q}\circ\rho\circ\eta_{\mathcal{A}}=q^{\prime}\circ\Delta_{WAP}\circ\rho\circ\eta_{\mathcal{A}}=q^{\prime}\circ{id}_{F(\mathcal{A})}\circ\eta_{\mathcal{A}}=id_{\mathcal{A}}.
	\end{equation*}
	Hence the proof is complete.
\end{proof}
\begin{cor}\label{Cor2.7}
	If $\mathcal{A}$ is a dual Banach algebra and $F(\mathcal{A})$ is Connes amenable, then $\mathcal{A}$ is Connes amenable.	
\end{cor}
\begin{proof}
	If $F(\mathcal{A})$ is Connes amenable, then by Corollary \ref{c2.5}, $F(\mathcal{A})$ is $WAP$-biprojective and has a unit. Applying Theorem \ref{Th 5.5} (ii) and \cite[Lemma 2.7]{Daws:2007}, $\mathcal{A}$ is Connes biprojective and has a unit. So $\mathcal{A}$ is Connes amenable \cite[Theorem 2.2]{Shi:2016}.
\end{proof}

\begin{Remark}
	Daws \cite[Lemma 2.7]{Daws:2007} showed that $F(\mathcal{A})$ is unital if and only if $\mathcal{A}$ is unital, where $\mathcal{A}$ is a dual Banach algebra. We show that if $\mathcal{A}$ is a Banach algebra with a bounded approximate identity, then $F(\mathcal{A})$ has a unit (without duality condition on $\mathcal{A}$). This statement helps us to figure out $WAP$-biprojectivity of the enveloping dual Banach algebras associated to certain Banach algebras. 
\end{Remark}
\begin{lemma}\label{l5.9}
	If $\mathcal{A}$ is a Banach algebra with a bounded approximate identity, then $F(\mathcal{A})$ has a unit.
\end{lemma}
\begin{proof}
	Let $(e_{\alpha})$ be a bounded approximate identity in $\mathcal{A}$. Regard $(\bar{e}_{\alpha})=(\eta_{\mathcal{A}}({e}_{\alpha}))$ as a bounded net in $F(\mathcal{A})$, where $\eta_{\mathcal{A}}:\mathcal{A}\longrightarrow{F}(\mathcal{A})$. By Banach-Alaoglu Theorem $(\bar{e}_{\alpha})$ has a $wk^{\ast}$-limit point in $F(\mathcal{A})$. Define $\Phi_{0}=wk^{\ast}\hbox{-}\lim\limits_{\alpha}\bar{e}_{\alpha}$. We claim that $\Phi_{0}$ is a unit for $F(\mathcal{A})$. For every $a\in{\mathcal{A}}$ and $\lambda\in{WAP(\mathcal{A}^{\ast})}$, we have
	\begin{equation*}
	\begin{split}
	\langle\lambda,a\cdot\Phi_{0}\rangle&=\langle\lambda\cdot{a},\Phi_{0}\rangle=\lim\limits_{\alpha}\langle\lambda\cdot{a},\bar{e}_{\alpha}\rangle=\lim\limits_{\alpha}\langle{e}_{\alpha},\lambda\cdot{a}\rangle\\&=\lim\limits_{\alpha}\langle{a}{e}_{\alpha},\lambda\rangle=\langle{a},\lambda\rangle=\langle\lambda,\bar{a}\rangle.
	\end{split}
	\end{equation*}
	So $a\cdot\Phi_{0}=\bar{a}$. By similarity for the right action, $\Phi_{0}\cdot{a}=\bar{a}$. Since $\eta_{\mathcal{A}}$ has a $wk^{\ast}$-dense range, for every $\Psi\in{F(\mathcal{A})}$ there exists a bounded net $(a_{\alpha})$ in $\mathcal{A}$ such that $\Psi=wk^{\ast}\hbox{-}\lim\limits_{\alpha}\bar{a}_{\alpha}$ in ${F(\mathcal{A})}$. Since ${F(\mathcal{A})}$ is a dual Banach algebra \cite[Theorem 4.10]{Runde:2004}, the multiplication in ${F(\mathcal{A})}$ is separately $wk^{\ast}$-continuous \cite[Exercise 4.4.1]{Runde:2002}. Lemma \ref{l5.3} (i) implies that
	\begin{equation*}
	\Psi\square\Phi_{0}=wk^{\ast}\hbox{-}\lim\limits_{\alpha}(\bar{a}_{\alpha}\square\Phi_{0})=wk^{\ast}\hbox{-}\lim\limits_{\alpha}({a}_{\alpha}\cdot\Phi_{0})=wk^{\ast}\hbox{-}\lim\limits_{\alpha}\bar{a}_{\alpha}=\Psi,
	\end{equation*}    
	 similarly $\Phi_{0}\square\Psi=\Psi$.
\end{proof}
\begin{cor}
Let $\mathcal{A}$ be a Banach algebra with a bounded approximate identity. Then $F(\mathcal{A})$ is $WAP$-biprojective if and only if $F(\mathcal{A})$ is Connes biprojective.
\end{cor}
\begin{proof}
Since $\mathcal{A}$ has a bounded approximate identity, by Lemma \ref{l5.9}, $F(\mathcal{A})$ has a unit. Applying Corollary \ref{c2.5} and \cite[Theorem 2.2]{Shi:2016}, $F(\mathcal{A})$ is $WAP$-biprojective if and only if $F(\mathcal{A})$ is Connes amenable if and only if $F(\mathcal{A})$ is Connes biprojective. 
	\end{proof}
\begin{cor}\label{C2.11}
	If $\mathcal{A}$ is a reflexive Banach algebra and $F(\mathcal{A})$ is $WAP$-biprojective, then $F(\mathcal{A})$ is Connes biprojective.
\end{cor}
\begin{proof}
In a reflexive Banach space, by Banach-Alaoglu theorem every bounded sequence has a weakly convergence subsequence. One can see that $WAP(\mathcal{A}^*)=\mathcal{A}^*$. Hence $F(\mathcal{A})=\mathcal{A}^{**}=\mathcal{A}$. Applying Theorem \ref{Th 5.5} (ii), $F(\mathcal{A})=\mathcal{A}$ is Connes biprojective.
\end{proof}
\begin{Proposition}
	For a locally compact group $G$, if $F(M(G))$ is $WAP$-biprojective, then $G$ is amenable.
\end{Proposition}
\begin{proof}
 Suppose that $F(M(G))$ is $WAP$-biprojective. Since $M(G)$ has a unit, by Lemma \ref{l5.9} and Corollary \ref{c2.5}, $F(M(G))$ is Connes amenable. Then by Corollary \ref{Cor2.7}, $M(G)$ is Connes amenable. So $G$ is amenable \cite[Theorem 5.4]{Runde:2003}.	
\end{proof}

  Zhang showed that the Banach algebra $\ell^{2}(X)$ with the pointwise multiplication is not biprojective, where $X$ is an infinite set \cite[\textsection2]{zhan:99}. We extend this example to the $WAP$-biprojective case:
\begin{Proposition}\label{p14}
  Let $X$ be an infinite set. Then $F(\ell^2(X))$ is not $WAP$-biprojective.
\end{Proposition}
\begin{proof}
Since $\mathcal{A}=\ell^2(X)$ is a Hilbert space, by a similar argument as in the Corollary \ref{C2.11}, we have $F(\mathcal{A})=\mathcal{A}$. We show that $\mathcal{A}$ is not $WAP$-biprojective. Suppose conversely that $\rho :{F}(\mathcal{A})\longrightarrow{F}_{\mathcal{A}}(\mathcal{A}\hat{\otimes}\mathcal{A})$ is a $wk^{\ast}$-$wk^{\ast}$ continuous $\mathcal{A}$-bimodule morphism such that $\Delta_{WAP}\circ\rho=id_{F(\mathcal{A})}$.  For every $i\in{X}$ consider $\rho(e_i)$, where $e_i$ is the element of $\mathcal{A}$ equal to $1$ at $i$ and $0$ elsewhere. Since $\eta_{\mathcal{A}\hat{\otimes}\mathcal{A}}:\mathcal{A}\hat{\otimes}\mathcal{A}\longrightarrow F_{\mathcal{A}}(\mathcal{A}\hat{\otimes}\mathcal{A})$ has a $wk^{\ast}$-dense range, there exists a bounded net $(u_{\alpha})$ in $\mathcal{A}\hat{\otimes}\mathcal{A}$ such that  $\rho(e_i)=wk^{\ast}\hbox{-}\lim\limits_{\alpha}\bar{u}_{\alpha}$. Since $\rho(e_i)=e_i\cdot\rho(e_i)\cdot e_i$, one can see that $\rho(e_i)=wk^{\ast}\hbox{-}\lim\limits_{\alpha}e_i\cdot\bar{u}_{\alpha}\cdot e_i=wk^{\ast}\hbox{-}\lim\limits_{\alpha}\lambda_{\alpha}\overline{{e}_i\otimes{e}_i}$ for some $(\lambda_{\alpha})\subseteq{\mathbb{C}}$. Since $\Delta_{WAP}$ is $wk^*$-continuous, 
\begin{equation*}
{e}_i=\Delta_{WAP}\circ\rho(e_i)=wk^{\ast}\hbox{-}\lim\limits_{\alpha}\lambda_{\alpha}\Delta_{WAP}(\overline{{e}_i\otimes{e}_i})=wk^{\ast}\hbox{-}\lim\limits_{\alpha}\lambda_{\alpha}\pi({e}_i\otimes{e}_i)=wk^{\ast}\hbox{-}\lim\limits_{\alpha}\lambda_{\alpha}{e}_i.
\end{equation*}
So  $\lambda_{\alpha}\overset{\vert\cdot\vert}{\longrightarrow}1$ in $\mathbb{C}$. So $\rho(e_i)=\overline{{e}_i\otimes{e}_i}$. Consider the identity operator $I:\mathcal{A}\rightarrow\mathcal{A}$, which can be viewed as an element of $(\mathcal{A}\hat{\otimes}\mathcal{A})^*$ \cite[\textsection3]{Daws:2006}. Define the map $\Phi:\mathcal{A}\hat{\otimes}\mathcal{A}\longrightarrow\mathcal{A}$ by $\Phi(a\otimes{b})=aI(b)$. We claim that $\Phi$ is weakly compact. We know that the unit ball of $\mathcal{A}\hat{\otimes}\mathcal{A}$ is the closure of the convex hull of $\{a\otimes{b} \ \colon \ \Vert{a}\Vert=\Vert{b}\Vert\leq1\}$. Since in a reflexive Banach space every bounded set
is relatively weakly compact, the set $\{ab \ \colon \ \Vert{a}\Vert=\Vert{b}\Vert\leq1\}$ is relatively weakly compact. So $\Phi$ is weakly compact. Applying \cite[Lemma 3.4]{Daws:2006}, we have $I\in{WAP(\mathcal{A}\hat{\otimes}\mathcal{A})^*}$. If $x=\sum\limits_{i\in{X}}\beta_{i}e_{i}$ is an element in $\mathcal{A}$, then $\rho(x)=\sum\limits_{i\in{X}}\beta_{i}\overline{{e}_i\otimes{e}_i}$. So
\begin{equation}\label{e3.1}
\langle{I},\rho(x)\rangle=\sum\limits_{i\in{X}}\beta_{i}\langle I,{e}_i\otimes{e}_i\rangle=\sum\limits_{i\in{X}}\beta_{i}\langle I(e_i),e_i\rangle=\sum\limits_{i\in{X}}\beta_{i}.
\end{equation}
We have 
\begin{equation*}
\vert\langle{I},\rho(x)\rangle\vert\leq\Vert{I}\Vert\Vert\rho\Vert\Vert{x}\Vert<\infty.
\end{equation*}
So by (\ref{e3.1}), $\sum\limits_{i\in{X}}\beta_{i}$ converges for every $x=\sum\limits_{i\in{X}}\beta_{i}e_{i}$ in ${\mathcal{A}}$. Then $\ell^2(X)\subset\ell^1(X)$, which is a contradiction with \cite[Proposition 6.11]{Fol:99}.	
\end{proof}
\begin{Remark}\label{R15}
	Let $H$ and $H^\prime$ be reflexive Banach algebras and suppose that $i:H\rightarrow H^\prime$ is an isomorphism. In a reflexive Banach space, by Banach-Alaoglu theorem every bounded sequence has a weakly convergence subsequence. So $WAP(H^*)=H^*$ and $WAP({H^\prime}^*)={H^\prime}^*$. Thus $F(i)=(i^*\vert_{WAP({H^\prime}^*)})^*=i^{**}$, where $F(-)$ is the functor from Banach algebras into the dual Banach algebras \cite[Remark 2.10]{Choi:2014}. Since $H^{**}=H$ and ${H^\prime}^{**}=H^{\prime}$, $F(i)=i$. So $i$ is a $wk^*$-continuous map.
\end{Remark}
\begin{Remark}
	Let $G$ be a locally compact group. Rickert showed that $L^2(G)$
 is a Banach algebra with convolution if and only if $G$ is compact \cite{Rick:68}.
\end{Remark}
\begin{Theorem}
	Let $G$ be an infinite commutative compact group. Then the Banach algebra $F(L^2(G))$ is not $WAP$-biprojective.
\end{Theorem}
\begin{proof}
  By Plancherel's Theorem \cite[Theorem 1.6.1]{Rud:62}, $L^2(G)$ is isometrically isomorphic to $\ell^2(\Gamma)$, where $\Gamma$ is the dual group of $G$ and $\ell^2(\Gamma)$ is a Banach algebra with pointwise multiplication. Since $L^2(G)$ and $\ell^2(\Gamma)$ are Hilbert spaces, by Remark \ref{R15}, this isomorphism should be a $wk^*$-continuous map. By Proposition \ref{p14}, $\ell^2(\Gamma)$ is not $WAP$-biprojective. So $F(L^2(G))=L^2(G)$ is not $WAP$-biprojective.	
\end{proof}

\section{Examples}
The semigroup $S$ is weakly left (respectively, right) cancellative
if $s^{-1}F=\{x\in{S}:sx\in{F}\}$ (respectively, $Fs^{-1}=\{x\in{S}:xs\in{F}\}$) is finite for every $s\in{S}$ and every finite subset $F$ of $S$, and $S$ is weakly cancellative if it is both weakly
left cancellative and weakly right cancellative \cite[Definition 3.14]{Dales:2010}. 
\begin{Example}
	Let $S$ be the set  of natural numbers $\mathbb{N}$ with the binary operation $(m,n)\longmapsto\max\{m,n\}$, where $m$ and $n$ are in $\mathbb{N}$. Then $S$ is a weakly cancellative semigroup \cite[Example 3.36]{Dales:2010}. So $\ell^1(S)$ is a dual Banach algebra with the predual $c_{0}(S)$ \cite[Theorem 4.6]{Dales:2010}. Clearly $S$ is unital but it is not a group, so $\ell^1(S)$ is not Connes amenable \cite[Theorem 5.13]{Daws:2006}. Moreover $F(\ell^1(S))$ is not Connes amenable \cite[\textsection7.1]{Daws:2007}. Since $\ell^1(S)$ has a unit, by \cite[Lemma 2.7]{Daws:2007}, $F(\ell^1(S))$ has a unit. Applying Corollary \ref{c2.5},
	$F(\ell^1(S))$ is not $WAP$-biprojective.  
	
	Note that if we consider this semigroup with the binary operation $(m,n)\longmapsto\min\{m,n\}$, where $m$ and $n$ are in $\mathbb{N}$. Since $S$ is not a weakly cancellative semigroup, $\ell^1(S)$ is not a dual Banach algebra \cite[Theorem 4.6]{Dales:2010}. Moreover $F(\ell^1(S))$ is not Connes amenable \cite[Theorem 7.6]{Daws:2007}. Also $\ell^1(S)$ has a bounded approximate identity $(\delta_n)_{n\geq1}$, where $\delta_n$ is the characteristic function of $\{n\}$. By Lemma \ref{l5.9} and Corollary \ref{c2.5},
	$F(\ell^1(S))$ is not $WAP$-biprojective.
\end{Example}
\begin{Example}\label{ex3.2}
	Let $\mathcal{A}$ be a Banach space. Suppose that $\Lambda$ is a non-zero linear functional on $\mathcal{A}$ with $\Vert\Lambda\Vert\leq1$. Define $a\cdot b=\Lambda(a)b$ for every $a,b\in\mathcal{A}$. One can easily show that $(\mathcal{A},\cdot)$ is a Banach algebra and $\Delta(\mathcal{A})=\{\Lambda\}$. We show that the following statements hold:
	\begin{enumerate}
\item[(i)] Consider $x_{0}\in{\mathcal{A}}$ such that $\Lambda(x_{0})=1$, define a map $\psi:\mathcal{A}\longrightarrow\mathcal{A}\hat{\otimes}\mathcal{A}$ by $\psi(a)=x_{0}\otimes{a}$, where $a\in{\mathcal{A}}$. One can see that $\psi$ is a bounded $\mathcal{A}$-bimodule morphism and $\pi_\mathcal{A}\circ\psi=id_{\mathcal{A}}$. So $\mathcal{A}$ is biprojective and Theorem \ref{Th 5.5} (i) implies that $F(\mathcal{A})$ is $WAP$-biprojective.
\item[(ii)] We show that $F(\mathcal{A})=\mathcal{A}^{\ast\ast}$, to see this for every $\psi\in{\mathcal{A}^{\ast}}$ and $a\in{\mathcal{A}}$, the map $\mathcal{A}\rightarrow\mathcal{A}^{\ast}$, $a\mapsto\psi\cdot{a}$ is weakly compact. For every $b\in{\mathcal{A}}$ we have
\begin{equation*}
\langle{b},\psi\cdot{a}\rangle=\langle{a\cdot b},\psi\rangle=\langle{b},\Lambda(a)\psi\rangle.
\end{equation*}
Let $\{a_{n}\}$ be a bounded sequence in $\mathcal{A}$. Since $\Lambda$ is a bounded linear functional on $\mathcal{A}$, $\{\Lambda({a}_{n})\}$ is a bounded sequence in $\mathbb{C}$. So there exists a convergence subsequence $\{\Lambda({a}_{{n}_{k}})\}$ in  $\mathbb{C}$. Thus $\{\Lambda({a}_{{n}_{k}})\}\psi$ converges in $\mathcal{A}^{\ast}$. So $\psi\cdot{a_{{n}_{k}}}$ converges weakly in $\mathcal{A}^{\ast}$.
Applying \cite[Lemma 5.9]{Choi:2014}, $WAP(\mathcal{A}^{\ast})=\mathcal{A}^{\ast}$. Therefore $F(\mathcal{A})=\mathcal{A}^{\ast\ast}$. Also $\mathcal{A}$ is an Arens regular Banach algebra \cite[Theorem 1.4.11]{Pal:94}. 	
\item[(iii)] We claim that $F(\mathcal{A})$ is Connes amenable if and only if $\dim(\mathcal{A})=1$. If $F(\mathcal{A})=\mathcal{A}^{\ast\ast}$ is Connes amenable, then it has a unit. So $\mathcal{A}$ has a bounded approximate identity $(e_\alpha)$ \cite[Proposition 2.9.16 (iv)]{Dales:2000}. We have
\begin{equation*}
x_{0}=\lim\limits_{\alpha}x_{0}{e_\alpha}=\lim\limits_{\alpha}\Lambda(x_{0}){e_\alpha}=\lim\limits_{\alpha}{e_\alpha}.
\end{equation*}
Thus $\mathcal{A}$ has a unit and since for every $b\in{\mathcal{A}}$, $b=bx_{0}=\Lambda(b)x_{0}$, $\dim(\mathcal{A})=1$.\\ Conversely if $\dim(\mathcal{A})=1$, then $\mathcal{A}\cong\mathbb{C}$ as Banach algebra. So $F(\mathcal{A})\cong\mathbb{C}$ is Connes amenable.
\end{enumerate}
\end{Example}
\begin{Example}
	Set $\mathcal{A}=\left(\begin{array}{cc} 0&\mathbb{C}\\
	0&\mathbb{C}\\
	\end{array}
	\right)$. With the usual matrix multiplication and $\ell^{1}$-norm, $\mathcal{A}$ is a Banach algebra. Since $\mathbb{C}$ is a dual Banach algebra, $\mathcal{A}$ is a dual Banach algebra. Moreover $\mathcal{A}$ is a Hilbert space. By a similar argument as in Corollary \ref{C2.11} we have $F(\mathcal{A})=\mathcal{A}^{**}=\mathcal{A}$. Since $\mathcal{A}$ has a right identity but it does not have an identity, $F(\mathcal{A})=\mathcal{A}$
	is not Connes-amenable. We define a map $\tau:\mathcal{A}\longrightarrow\mathcal{A}\hat{\otimes}\mathcal{A}$ by $\left(\begin{array}{cc} 0&x\\
	0&y\\
	\end{array}
	\right)\longmapsto\left(\begin{array}{cc} 0&x\\
	0&y\\
	\end{array}
	\right)\otimes\left(\begin{array}{cc} 0&1\\
	0&1\\
	\end{array}
	\right)$. It is easy to see that $\tau$ is a bounded $\mathcal{A}$-bimodule morphism and also it is a right inverse for $\pi_{\mathcal{A}}$. So $\mathcal{A}$ is biprojective. By Theorem \ref{Th 5.5} (ii), $F(\mathcal{A})$ is $WAP$-biprojective and also Corollary \ref{C2.11} implies that it is Connes biprojective.	
\end{Example}

\begin{Example}
	Consider the Banach algebra $\ell^1$ of all sequences $a=(a(n))$ of complex numbers with 
	\begin{equation*}
	\Vert{a}\Vert:=\sum\limits_{n=1}^\infty\vert a(n)\vert<\infty,
	\end{equation*}
	and the following product
	\begin{equation*}
	(a\ast b)(n)=\left\{
	\begin{array}{ll}
	a(1)b(1)& \hbox{if}\quad n=1\\
	a(1)b(n)+b(1)a(n)+a(n)b(n)&\hbox{if}\quad n>1
	\end{array}
	\right.
	\end{equation*}
	for every $a,b\in{\ell^1}$. By similar argument as in \cite[Example 4.1]{Sha:17}, $(\ell^1,*)$ is a dual Banach algebra with respect to $c_0$. We claim that $(\ell^1,*)$ is not Connes amenable. Suppose conversely that $(\ell^1,*)$ is Connes amenable. Define $\varphi_1:\ell^1\longrightarrow\mathbb{C}$ by $\varphi_1(a)=a(1)$ for every $a\in{\ell^1}$. It is easy to see that $\varphi_1$ is a $wk^*$-continuous character on $(\ell^1,*)$. Using \cite[Theorem 2.2]{Mahmoodi:2014}, $(\ell^1,*)$ is $\varphi_1$-Connes amenable. Similar argument as in \cite[Example 4.1]{Sha:17} leads us to a contradiction. So $(\ell^1,*)$ is not Connes amenable and Corollary \ref{Cor2.7} implies that $F(\ell^1)$ is not Connes amenable. it is easy to see that $(\ell^1,*)$ has a unit $\delta_1$, where $\delta_1$ equal to $1$ at $n=1$ and $0$ elsewhere. So by Lemma \ref{l5.9}, $F(\ell^1)$ has a unit. Thus by Corollary \ref{c2.5}, $F(\ell^1)$ is not $WAP$-biprojective. 
\end{Example}
\begin{small}
	
\end{small}


\begin{thebibliography}{99}
	\bibitem{Choi:2014} Y. Choi, E. Samei and R. Stokke; {\it Extension of derivations, and Connes-amenability of the enveloping dual Banach	algebra}, Math. Scand. {\bf 117} (2015), 258-303.
	\bibitem{Dales:2000} H. G. Dales; {\it Banach algebras and automatic continuity}, Clarendon Press, Oxford, 2000.
    \bibitem{Dales:2010} H. G. Dales, A. T. M. Lau and D. Strauss; {\it Banach algebras on semigroups and their compactifications}, Memoir Amer. Math. Soc. {\bf 966} (2010).
    \bibitem{Daws:2006} M. Daws; {\it Connes-amenability of bidual and weighted semigroup algebras}, Math. Scand. {\bf 99} (2006), 217-246.
	\bibitem{Daws:2007} M. Daws; {\it Dual Banach algebras: representations and injectivity}, Studia Math. {\bf 178} (2007), 231-275.
   	 \bibitem{Fol:99} G. B. Folland; {\it Real analysis: modern techniques and their applications}, Wiley-Interscience Pub, New York, 1999.
   	\bibitem{Lau:97} A. T. M. Lau and R. J. Loy; {\it Weak amenability of Banach algebras on locally compact groups}, J. Funct. Anal. {\bf 145} (1997), 175-204.
   	\bibitem{Mahmoodi:2014} A. Mahmoodi; {\it On $\phi$-Connes amenability for dual Banach algebras}, J. Linear Topol. Algebra. {\bf 03} (2014), 211-217.
     \bibitem{Rob:98} R. E. Megginson; {\it An introduction to
  	Banach space theory}, Springer-Verlag, New York, 1998.
  \bibitem{Pal:94} T. W. Palmer; {\it Banach algebras and the general theory of $\ast$-algebras}, Vol. I, Cambridge University Press, Cambridge, 1994.
  \bibitem{Rick:68} N. W. Rickert; {\it Convolution of $L^2$-functions}, Colloq. Math. {\bf 19} (1968), 301-303.
 \bibitem{Rud:62} W. Rudin; {\it Fourier analysis on groups}, Interscience Publishers, New York - London, 1962.
   	\bibitem{Runde:2001} V. Runde; {\it Amenability for dual Banach algebras}, Studia Math. {\bf 148} (2001), 47-66.
      \bibitem{Runde:2003} V. Runde; {\it Connes-amenability and normal, virtual diagonals for measure algebras I}, J. London Math. Soc. {\bf 67} (2003), 643-656.	
      \bibitem{Runde:2004} V. Runde; {\it Dual Banach algebras: Connes-amenability, normal, virtual diagonals, and injectivity of the predual bimodule}, Math. Scand. {\bf 95} (2004), 124-144.
   \bibitem{Runde:2002} V. Runde; {\it Lectures on amenability}, Lecture Notes in Mathematics, Vol. 1774, Springer-Verlag, Berlin, 2002.
     \bibitem{Sha:17} S. F. Shariati, A. Pourabbas and A. Sahami; {\it Johnson pseudo-Connes amenability of dual Banach algebras}, see https://arxiv.org/abs/1801.03369v2 (2018).
   \bibitem{Shi:2016} A. Shirinkalam and A. Pourabbas; {\it Connes-biprojective dual Banach algebras}, U.P.B. Sci. Bull. Series A. {\bf 78}. Iss. 3 (2016), 174-184.
    \bibitem{zhan:99} Y. Zhang; {\it Nilpotent ideals in a class of Banach algebras}, Proc. Amer. Math. Soc. {\bf 127} (1999), 3237-3242.
\end{thebibliography}
\end{document}